\documentclass[letterpaper, 10 pt, journal, twoside]{IEEEtran}  

\IEEEoverridecommandlockouts                              

\usepackage{amsmath,amssymb,amsthm} 

\newtheoremstyle{theoremdd}
{\topsep}
{\topsep}
{\itshape}
{0pt}
{\bfseries}
{:}
{ }
{\thmname{#1}\thmnumber{ #2}\thmnote{ (#3)}}

\theoremstyle{theoremdd}

\newtheorem{theorem}{Theorem}[section]

\newtheorem{remark}[theorem]{Remark}

\usepackage{tikz}

\usepackage[ruled,vlined]{algorithm2e}

\usepackage{mathrsfs}
\usepackage{graphicx} 
\usepackage{epsfig} 
\usepackage{epstopdf}
\usepackage{color}
\usepackage{float}
\usepackage{mathtools}

\usepackage{cite}
\usepackage[ruled,vlined]{algorithm2e}
\usepackage[font=footnotesize, justification=centering ]{caption}

\usepackage[appendix=append]{apxproof}
\newtheoremrep{theorem}{Theorem}[section]

\newtheoremrep{lemma}[theorem]{Lemma}
\newtheoremrep{corollary}[theorem]{Corollary}

\usepackage{tikz}
\usepackage{pgfplots}
\usetikzlibrary{shadows}
\usepackage[edges]{forest}
\usetikzlibrary{shapes,arrows,positioning,calc,chains}
\pgfplotsset{
	standard/.style={
		axis x line=middle,
		axis y line=middle,
		enlarge x limits=0.15,
		enlarge y limits=0.15,
		every axis x label/.style={at={(current axis.right of origin)},anchor=north west},
		every axis y label/.style={at={(current axis.above origin)},anchor=north east},
		every axis plot post/.style={mark options={fill=white}}
	}
}

\makeatletter
\DeclareRobustCommand{\rvdots}{%
	\vbox{
		\baselineskip4\p@\lineskiplimit\z@
		\kern-\p@
		\hbox{.}\hbox{.}\hbox{.}
}}
\makeatother

\usepackage{lipsum}
\makeatletter
\def\footnoterule{\relax%
	\kern-5pt
	\hbox to \columnwidth{\hfill\vrule width .9\columnwidth height 0.4pt\hfill}
	\kern4.6pt}
\makeatother

\usepackage{color,hyperref}
\definecolor{darkblue}{rgb}{0.0,0.0,0.5}
\hypersetup{colorlinks,breaklinks,linkcolor=darkblue,urlcolor=darkblue,anchorcolor=darkblue,citecolor=darkblue}

\makeatother
\title{On Robustness of Double Linear \\  Trading with Transaction Costs}

\author{\large Chung-Han Hsieh,$^{*}$ \textit{Member, IEEE} 
	\thanks{This paper is partially supported by the Ministry of Science and Technology~(MOST), Taiwan, R.O.C. under Grant:  MOST 111--2221--E--007--124-- \\${}^*$Chung-Han Hsieh is with the Department of Quantitative Finance, National Tsing Hua University, Hsinchu, Taiwan 30013, R.O.C. E-mail: \href{mailto: ch.hsieh@mx.nthu.edu.tw}{ch.hsieh@mx.nthu.edu.tw}. } 
}

\begin{document}

	\maketitle
	\thispagestyle{empty}
	\pagestyle{empty}
	
	
	\begin{abstract} 
	A trading system is said to be \textit{robust} if it generates a robust return regardless of market direction. 	To this end, a consistently positive expected  trading gain is often used as a robustness metric for a trading system.
	In this paper, we propose a new class of trading policies called the \textit{double linear policy} in an asset trading scenario when the transaction costs are involved. 
	Unlike many existing papers, we first show that  the desired robust positive expected gain may disappear when transaction costs are involved. 
	Then we quantify under what conditions  the desired  positivity can still be preserved. 
	In addition, we conduct heavy Monte-Carlo simulations for an underlying asset whose prices are governed by a geometric Brownian motion with jumps to validate our theory. 
	A more realistic backtesting example involving historical data for cryptocurrency Bitcoin-USD  is also~studied.
	\end{abstract}
	
	
	
	%
	\section{Introduction} \label{SECTION: INTRODUCTION}
	A trading system is said to be \textit{robust} if it generates a robust return regardless of market direction. In particular, a consistently positive expected  trading gain-loss is often used as a robustness metric.
	It is known that the  so-called  \textit{Simultaneous Long-Short (SLS)} control scheme, see~\cite{barmish2008trading,
		barmish2011performance, 	 barmish2011arbitrage,
		barmish2015new, 
		baumann2016stock, baumann2017simultaneously}, 
	guarantees the satisfaction of the so-called \textit{Robust Positive Expectation} (RPE); i.e., the cumulative trading gain-loss function is guaranteed to have positive expected value for a broad class of stock price processes. 
	This fact attracts many new extensions and modifications to the SLS theory.

	For example, the SLS theory involves RPE results with respect to stock prices having time-varying drift and volatility; see~\cite{primbs2017robustness}, prices generated from Merton’s diffusion model~\cite{baumann2017simultaneously}, generalization from a static linear policy to the case of Proportional–Integral (PI) controllers~\cite{malekpour2018generalization}, and discrete-time systems with delays~\cite{malekpour2016stock}. 
	More recently, in~\cite{deshpande2020simultaneous}, an SLS control with cross-coupling is proposed to trade two stocks.
	In~\cite{o2020generalized}, a generalized RPE Theorem to the case of an SLS controller which can have different parameters for the long and short sides of the trade is studied. 
	In~\cite{maroni2019robust},  an~$\mathcal{H}_\infty$  approach for selecting the SLS controller parameters is proposed.
	
	To close this brief introduction, we mention some related studies using the control-theoretic approach in stock trading scenario; e.g., see \cite{zhang2001stock} for studying optimal selling rule, \cite{tie2018optimal, deshpande2018generalization} for studying optimal pair trading, and~\cite{barmish2021feedforward} for studying a new feedforward control in transaction level price model.
	
	While various extensions and ramifications are proposed to improve the SLS theory, the effect of transaction costs, which is believed to erode the desired positivity, is arguably lacking in the  existing study. 
	To this end, we propose a new SLS-based trading structure, which we call the \textit{double linear policy}, in a discrete-time setting to study  positive expectation property closely when the transaction costs are involved.
	
	The contributions of this paper are threefold: First,   we propose a new double linear policy scheme  for a class of asset price processes involving independent returns; see Section~\ref{SECTION: Problem Formulation}.
	In the financial market with nonzero transaction costs, we derive explicit expressions for the expected value and variance of the cumulative gain-loss function using the proposed policy scheme; see Section~\ref{SECTION: Analysis of Trading Performance}.
	Second, we go beyond the existing literature by quantifying under what condition  the desired expected positivity can still be preserved when transaction costs are involved.
	A novel asymptote approach is used to obtain a relatively conservative positivity; see Theorem~\ref{theorem: Robust Positive Expectation} in Section \ref{SECTION: Analysis of Trading Performance}. 
	Third, we   validate the theory via heavy Monte-Carlo simulations.  An empirical backtesting example using the historical data for cryptocurrency Bitcoin USD is also provided; see Section~\ref{SECTION: Illustrative Examples}.


	\section{Preliminaries}	\label{SECTION: Problem Formulation}
	In this section, we provide some necessary preliminaries that are useful for the analysis to follow. 

	\subsection{Asset Prices and Returns}
	For $k=0, 1, 2, \dots$, let~$S(k) > 0$ be the \textit{price} of an underlying risky asset, often to be stock, at stage~$k$. Then the associated \textit{per-period return}, call it $X(k)$, is given by 
	$
		X(k) := \frac{S(k+1) - S(k)}{S(k)}. 
$
	Assuming that~$X_{\min} \leq X(k)\leq X_{\max}$ for all $k$ almost surely with known bounds 
	$
	-1 < X_{\min} < 0 < X_{\max} < \infty
	$
	and~$X_{\min}$ and~$X_{\max}$ are in the support of $X(k)$.
	Unlike many existing papers in finance that models the asset price dynamics, we assume that~$X(k)$ are independent in $k$ and have \textit{arbitrary} distribution with common but unknown mean $\mathbb{E}[X(k)] := \mu  $ and \textit{unknown} variance 
	$
	{\rm var}(X(k)) = \mathbb{E}[(X(k) - \mu)^2]:= \sigma^2 \geq 0
	$
	for all $k$.

	\begin{remark} \rm
		In practice, since the true mean $\mu$ and variance~$\sigma$ are unknown, one often works with the estimated surrogates such as sample mean and sample standard deviation; see \cite[Chapter 9]{luenberger2013investment} and  Section~\ref{subsection: trading with historical data} for an illustration. 
		While this approach may incur some estimation errors, as seen in the sections to follow, we shall quantify a range for $\mu$ such that the desired positivity of expected return can be preserved.
	\end{remark}

	\subsection{Double Linear Policy and Transaction Costs}
	With initial account value $V(0):=V_0>0$, we split it into two parts as follows: 
	Taking an \textit{allocation constant}~$\alpha \in [0,1]$, we define~$V_L(0) := \alpha V_0$ as the initial account value  for \textit{long} position   and~$V_S(0) := (1-\alpha)V_0$ for \textit{short} position.\footnote{In stock trading, a long position means that the trader purchases shares of stock in the hope of making a profit from a subsequent rise in the price of the underlying stock. On the other hand, a short position means that the trader borrows shares from a broker in the hope of making a profit from a subsequent fall in the price of the stock; see~\cite{luenberger2013investment, bodie2009investments}.}
	Then~$V(0) = V_L(0) +V_S(0) .$ 
	The~\emph{trading policy} or  \textit{controller}~$\pi(\cdot)$ is given~by
	$
	\pi(k):=\pi_L(k) +\pi_S(k)
	$
	where~$\pi_L$ and~$\pi_S$ are of \textit{double linear} forms:
	\begin{align} \label{eq: double linear feedback}
		\begin{cases}
			\pi_L( k ) = w V_L\left( k \right);\\
			\pi_S( k ) =  - w V_S \left( k \right)
		\end{cases}
	\end{align}
	where the weight $w$ satisfies
	$	w \in \mathcal{W}$ for some admissible set $\cal W$.
	Specifically, 
	when there are transaction costs per trade with percentage rate $\varepsilon \in [0,1]$, we take $\varepsilon \pi_L(k) := \varepsilon w V_L(k)$ for long trade or~$\varepsilon |\pi_S(k)| := \varepsilon w V_S(k)$ for short trade. In the sequel, we take $\mathcal{W}:=[0, w_{\max}]$ with 
	$$
	w_{\max}:=\min \left\{ \frac{1}{1+\varepsilon},\, \frac{1}{X_{\max}+\varepsilon} \right\} .
	$$
	 In the sequel, we refer to the control scheme above as the \textit{double linear policy} with pair~$(\alpha, w)$. 

	\begin{remark} \rm \label{remark: admissible set}
		$(i)$ The double linear form can be viewed as a natural extension of the  pure long linear strategy by taking~$\alpha =1$; e.g., see~\cite{hsieh2019positive, hsieh2020necessary, hsieh2022generalization}.  
		$(ii)$ The choice of the upper bound for~$\cal W$ above corresponds to assuring \textit{survival} trades, see Lemma~\ref{lemma: survivability} in the next Section~\ref{subsection: Survivability} to follow, and the trades are \textit{cash-financed}; i.e.,~$|\pi(k)| \leq V(k)$ for all $k$ with probability one; see Lemma~\ref{lemma: Cash-Financing}. 
		As seen later in Section~\ref{SECTION: Analysis of Trading Performance} to follow, the double linear policy also enjoys some additional properties such as convexity and positive expectation.
	\end{remark}
	
	\subsection{Account Value  Dynamics}
	The associated account values under~$\pi_L$ and~$\pi_S$, denoted by~$V_L(k)$ and $V_S(k)$ respectively, are described by the following stochastic recursive equations:
	\begin{align} \label{eq: V_L and V_S}
	\begin{cases}
		{V_L}\left( {k + 1} \right) = V_L( k ) + X\left( k \right){\pi_L}\left( k \right) - \varepsilon \pi_L(k);\\
		{V_S}\left( {k + 1} \right) = V_S( k ) + X\left( k \right){\pi_S}\left( k \right) - \varepsilon |\pi_S(k)| .
	\end{cases} 
	\end{align}
	This implies that 
	$
	{V_L}( k ) = \prod_{j = 0}^{k - 1} {\left( {1 + w \widetilde{X}_L ( j )} \right)} {V_L}\left( 0 \right)$
	and
	$
	{V_S}( k ) = \prod_{j = 0}^{k - 1} {\left( {1 - w \widetilde{X}_S ( j )} \right)} {V_S}\left( 0 \right)
	$
	where $\widetilde{X}_L(k):=X(k) -\varepsilon$ and $\widetilde{X}_S(k) := X(k) + \varepsilon.$
	Therefore, the \textit{overall account value} at stage $k$ is given~by
	{\small \begin{align*} 
			V(k) 
			&= V_L(k) + V_S(k)  \\
			&= V_0\left(  \alpha {\prod_{j = 0}^{k - 1} \left( 1 + w \widetilde{X}_L( j ) \right)  + (1-\alpha)\prod_{j = 0}^{k - 1} \left( 1 - w \widetilde{X}_S( j ) \right) } \right).
		\end{align*} 
	}Note that the account value $V(k)$ depends on the allocation constant~$\alpha$, the decision weight $w$, transaction costs rate $\varepsilon $ and return sequence $X=\{X(i)\}_{i=0}^{k-1}$.
	
	\begin{remark} \rm
		 $(i)$ According to classical SLS theory; e.g., see~\cite{barmish2015new, barmish2011arbitrage}, the associated control law can be written as
		\begin{align*}
			\begin{cases}
				\pi_L(k) = \pi_0 +w(V_L(k)-V_L(0)); \\
				\pi_S(k) = -\pi_0 - w(V_S(k) - V_S(0))
			\end{cases}
		\end{align*}
		for some initial investment $\pi_0>0$. It is readily verified that the classical~SLS controller is equivalent to the double linear policy with $\alpha = 1/2$ if one takes $\pi_0 := wV_0/2$ and $V_L(0) = V_S(0) := V_0/2$.
		$(ii)$ Our framework can easily adapt to involve the risk-free asset, such as a \textit{bond} or \textit{bank account}, with per-period interest rate $r(k) \geq 0$ for all $k$ with probability one; see \cite{cvitanic2004introduction}. Then the account value dynamics for the long trade can be modeled as
		\begin{align*}
			V_L( {k + 1} ) 
			&= (1 + r(k) )V_L( k ) + (X( k ) - \varepsilon - r(k) )\pi_L(k)
		\end{align*}
		and the account value dynamics $V_S(k)$  for short trade stays the same.  By taking the risk-free asset as a \textit{cash} which yields~$r(k) :=0$ for all $k$, we replicate Equation~(\ref{eq: V_L and V_S}).
	\end{remark}

	\section{Analysis of Trading Performance} \label{SECTION: Analysis of Trading Performance}
	In this section, we study the trading performance of the proposed double linear policy. Specifically, we provide several results such as survivability, convexity, positive expectation, and cash-financing lemma. 
	
	\subsection{Survivability Considerations} \label{subsection: Survivability}
	As mentioned in the previous subsection, every $w \in \mathcal{W}$  assures \textit{survivability}, i.e., bankruptcy avoidance. That is, for all $k$, the~$w$-value that can potentially lead to $V(k)<0$ is disallowed. 
	
	\begin{lemmarep}[Survivability] \label{lemma: survivability}
		Let $\varepsilon \in [0, 1]$.
		The double linear policy with a pair $(\alpha, w) \in (0,1] \times \mathcal{W}$ assures
		$
		V(k) > 0
		$  for~$k =0,1,\dots,$ with probability one.
	\end{lemmarep}
	
	\textit{Proof.} See Appendix.
	
	\begin{proof} 
		Note that for $k=0$, $V(0) = V_0>0$. 
		Fix an integer~$k > 0$. Observe that
		\begin{equation}
			\begin{aligned}
				&V(k) = V_L(k) + V_S(k) \\
				&= \left( \alpha \prod_{j = 0}^{k - 1} \left( {1 + w \widetilde{X}_L\left( j \right)} \right)  + (1-\alpha)\prod_{j = 0}^{k - 1} {\left( 1 - w \widetilde{X}_S( j ) \right)}  \right)V_0 \\
				&\geq \left( \alpha ( 1 + w (X_{\min}-\varepsilon)  )^k  + (1-\alpha)  ( 1 - w (X_{\max}+\varepsilon)  )^k  \right)V_0\\
				&:=V_{\min}^*(k).
			\end{aligned}  
		\end{equation}
		To complete the proof, it suffices to show that $V_{\min}^*(k) > 0.$
		Since  $w \in \mathcal{W}$ 
		and~$X_{\min}>-1$, it implies that
$$
			1 + w (X_{\min} -\varepsilon) 
			> 1 - w_{\max} (1 + \varepsilon) \geq 0
$$
		where the last inequality holds since~$w_{\max} \leq 1/(1+\varepsilon)$.
		Hence $(1+ w (X_{\min} - \varepsilon) )^k > 0$ for all $k$. 
		On the other hand, since $0\leq w \leq 1/(X_{\max} + \varepsilon)$, it follows that  
$
			1- w (X_{\max} + \varepsilon)\geq 0.
$
Hence, $(1- w (X_{\max} + \varepsilon))^k \geq 0$.
In combination with the results above and the fact that $\alpha \in (0,1]$, we conclude that $V_{\min}^*(k) > 0$ for all~$k$ with probability~one. 
	\end{proof}
	
%
%

\subsection{Expected Trading Gain-Loss and Variance}

	We take
$
\mathcal{G}_k(\alpha, w, X, \varepsilon ) := V(k)-V_0
$
to be the \textit{cumulative trading gain-loss function} up to stage~$k$ where~$X:=\{X(i)\}_{i=0}^{k-1}$ is the sequence of returns~$X(0), \dots,X(k-1)$.
Then the expected cumulative gain-loss is given by
$$
\overline{\mathcal{G}}_k(\alpha, w, \mu, \varepsilon):=\mathbb{E}[\mathcal{G}_k(\alpha, w,  X, \varepsilon)].
$$
	The following lemma states the closed-form expression of the expected value and variance of the cumulative gain-loss function when there are transaction costs.

	\begin{lemmarep}[Expected Trading Gain-Loss and Variance] \label{lemma: expected cumulative gain or loss and variance}
		Fix~$\varepsilon \in [0,1]$. For stage $k=0,1,\dots$, 
		the double linear policy with pair~$(\alpha, w) \in [0,1] \times \mathcal{W}$ yields expected cumulative gain-loss function
		\begin{align*}
			&\overline{\mathcal{G}}_k(\alpha, w, \mu, \varepsilon) \\
			&=  V_0 \left( \alpha \left( 1 + w (\mu- \varepsilon) \right)^k+ (1-\alpha)\left( 1 - w (\mu + \varepsilon)  \right)^k - 1 \right). 
		\end{align*}
		Moreover, the corresponding variance is given by
		\begin{align*}
			&{\rm var}(\mathcal{G}_k(\alpha, w, X, \varepsilon))\\
			& = 	  \alpha^2   \left( (1 + w (\mu - \varepsilon) )^2 +  w^2 \sigma^2   \right)^k \\
			&\;\;\;\; + (1-\alpha)^2  \left( (1 - w (\mu + \varepsilon) )^2 + w^2 \sigma^2  \right)^k \\
			& \;\;\;\; + 2\alpha (1-\alpha)  ( 1 + 2 w \varepsilon - w^2 ( (\mu^2 + \sigma^2)- \varepsilon^2)  )^k  \\
			& \;\;\;\; - 2\alpha \left( 1 + w (\mu- \varepsilon) \right)^k  - 2(1-\alpha)\left( 1 - w (\mu+ \varepsilon) \right)^k  +1  \\
			& \;\;\;\;  -  \left( \alpha \left( 1 + w (\mu- \varepsilon) \right)^k+ (1-\alpha)\left( 1 - w (\mu + \varepsilon)  \right)^k - 1 \right)^2.
		\end{align*}
		Of course, the standard deviation of cumulative trading gain or loss is ${\rm std}(\mathcal{G}_k(\alpha, w, X, \varepsilon)) =\sqrt{{\rm var}(\mathcal{G}_k(\alpha, w, X, \varepsilon))}$.
	\end{lemmarep}
	
		\textit{Proof.} See Appendix.
		
	\begin{proof} 
		Let $k$ be fixed. Set $R_L(k) := \prod_{j = 0}^{k - 1} \left( 1 + w \widetilde{X}_L(j) \right)$ and $R_S(k) := \prod_{j = 0}^{k - 1} \left( 1 - w \widetilde{X}_S(j)  \right).$
		Using the facts that the returns $X(k)$ are independent in $k$ and~$\mathbb{E}[X(k)] = \mu$ for all~$k$, we have
	{\small	
		\begin{align} \label{eq: expected trading gain }
				& \overline{\mathcal{G}}_k(\alpha, w, \mu, \varepsilon) \nonumber  \\
				& = V_0 ( \alpha \mathbb{E}[ R_L(k)]  + (1-\alpha) \mathbb{E}[R_S(k)]  - 1  ) \nonumber  \\
				&= V_0 \left( \alpha \left( 1 + w (\mu- \varepsilon) \right)^k+ (1-\alpha)\left( 1 - w (\mu + \varepsilon)  \right)^k - 1 \right).
		\end{align}
	}
		Subsequently,	to obtain the variance, we use the fact
		$
		{\rm var}(\mathcal{G}_k(\alpha, w, X, \varepsilon) )=\mathbb{E}[\mathcal{G}_k^2(\alpha, w, X, \varepsilon)]-\overline{\mathcal{G}}_k^2 (\alpha, w, \mu, \varepsilon).
		$
		Since the second term can be obtained from Equation~(\ref{eq: expected trading gain }), it remains to calculate the second moment $\mathbb{E}[\mathcal{G}_k^2(\alpha, w, X, \varepsilon)]$. 
		Without loss of generality, we set $V_0:=1$.
		Using the facts that~$X(k)$ are independent, $\mathbb{E}[X(k)] = \mu$, $\mathbb{E}[X^2(k)] = \sigma^2 + \mu^2$ for all $k$,  and the linearity of expected value, a straightforward calculation leads to
		\begin{align*}
			\mathbb{E}[\mathcal{G}_k^2(\alpha, w, X, \varepsilon)] 
			&=  \mathbb{E}[ \left( \alpha R_L(k)  + (1-\alpha)R_S(k)  - 1 \right)^2 ]\\
			&=   \alpha^2 \mathbb{E}[  R_L^2(k)]  + (1-\alpha)^2 \mathbb{E}[  R_S^2(k)] \\
			&\;\;\;\; + 2(1-\alpha)\alpha \mathbb{E}[  R_L(k)R_S(k)]  \\
			&\;\;\;\; - 2\alpha \mathbb{E}[  R_L(k)  ] - 2(1-\alpha)\mathbb{E}[  R_S(k)]   +1  
		\end{align*}
		where the three terms $\mathbb{E}[R_L^2(k)],$ $\mathbb{E}[R_S^2(k)]$, and $\mathbb{E}[R_L(k) R_S(k)]$ are obtained via lengthy but straightforward calculations. That is, we have
		\begin{align*}
			\mathbb{E}[R_L^2(k)] 
			&=    \left( (1 + w (\mu - \varepsilon) )^2 +  w^2 \sigma^2   \right)^k,
		\end{align*}
and
$
			\mathbb{E}[R_S^2(k)] 
			=   \left( (1 - w (\mu + \varepsilon) )^2 + w^2 \sigma^2  \right)^k,
$
		and the cross term
$
			\mathbb{E}[R_L(k)R_S(k)] 
			=     ( 1 + 2 w \varepsilon - w^2 ( (\mu^2 + \sigma^2)- \varepsilon^2)  )^k.
$
		Therefore, in combination with the results above, the desired expression for the variance is obtained.
%
		To complete the proof, we note that the standard deviation ${\rm std}(\mathcal{G}_k(\alpha, w, X, \varepsilon))=\sqrt{{\rm var}(\mathcal{G}_k(\alpha, w, X, \varepsilon))}$, which is desired. 
	\end{proof}

	\begin{remark} \rm
		$(i)$ As a sanity test, if one decides not to trade; i.e., $w=0$,  then
		$
		\overline{\mathcal{G}}_k(\alpha, w, \mu, \varepsilon) = {\rm var}(\mathcal{G}_k(\alpha, w, X, \varepsilon))=0.
		$
		$(ii)$~If there are no transaction costs and no uncertainty on the return; i.e., $\sigma = \varepsilon  :=0$, then  a straightforward calculation leads to
		$
		{\rm var}(\mathcal{G}_k(\alpha, w, X, \varepsilon)) =0.
		$	
		$(iii)$~If the returns has no trend; i.e., $\mu=0$, then it follows that $\overline{\mathcal{G}}_k(\alpha, w,\mu,\varepsilon) =V_0 \left(  \left( 1 - w \varepsilon \right)^k - 1 \right) \leq 0$. 
		Thus, with nonzero transaction costs~$\varepsilon>0$, one should invest with zero weight $w=0$ to avoid loss in expected gain-loss.
		%
		$(iii)$ Lemma~\ref{lemma: expected cumulative gain or loss and variance} above can be viewed as a generalization of the result for the classical discrete-time SLS controller considered in \cite{deshpande2018generalization}.
	\end{remark}

	\subsection{Convexity and Positive Expectation} \label{SECTION: Convexity}
To study the expected positivity, the following  convexity result of the expected gain-loss function is useful.

	\begin{lemmarep} [Convexity] \label{lemma: convexity} 
		Fix $\varepsilon \in [0,1]$. For $k > 1$,	if~$\alpha \in (0,1)$, then  the expected cumulative gain-loss function~$\overline{\mathcal{G}}_k(\alpha, w, \mu, \varepsilon)$ is strictly convex in $w \in \mathcal{W}$.
	\end{lemmarep}
	
		\textit{Proof.} See Appendix.
	
	\begin{proof} 
		According to Lemma~\ref{lemma: expected cumulative gain or loss and variance}, we have $\overline{\mathcal{G}}_k(\alpha, w, \mu, \varepsilon) = V_0 f_k(\alpha, w,\mu,\varepsilon)$  where
		\begin{align*}
			f_k(\alpha, w,\mu,\varepsilon) &:=  \alpha \left( 1 + w (\mu- \varepsilon) \right)^k \\
			&+ (1-\alpha)\left( 1 - w (\mu + \varepsilon)  \right)^k - 1 . 
		\end{align*}
		Note that any strict convex function multiplied by a strictly positive scalar is still strict convex.
		Fix $k \geq 2$.
		Since~$V_0>0$,  it suffices to show that $f_k(\alpha, w,\mu,\varepsilon)$ is strictly convex in $w \in \mathcal{W}$.
		Taking the derivative with respect to $w$ twice yields
		the second derivative 
		{\small \begin{align*}
			\frac{\partial^2 f_k(\alpha, w, \mu, \varepsilon) }{\partial w^2} 
			&= \alpha k(k-1)^2(1 + w (\mu - \varepsilon) )^{k-2} (\mu - \varepsilon)^2 \\
			&+ (1-\alpha)k(k-1)^2(1 - w (\mu + \varepsilon) )^{k-2}(\mu + \varepsilon)^2.
		\end{align*}
	 }To establish the desired convexity, it suffices to verify positivity of $(1 + w (\mu - \varepsilon) )^{k-2}$ and $(1 - w (\mu + \varepsilon) )^{k-2}$.
		Since $w \in \mathcal{W}$ and $\mu > -1$, we observe that
		$
		1+w (\mu - \varepsilon) > 1-w (1+\varepsilon)\geq 0.
		$  
		Therefore, $(1+w (\mu - \varepsilon))^{k-2} >0$.
		On the other hand, note that 
		$
		1 - w (\mu + \varepsilon)  \geq 1 - w(X_{\max} +\varepsilon)\geq 0,
		$ which implies that $(1 - w (\mu + \varepsilon) )^{k-2} \geq 0$.
		Therefore, in combination with the fact that $\alpha \in (0,1)$, it implies that $\frac{\partial^2}{\partial w^2} f_k(\alpha, w, \mu, \varepsilon) > 0$ for all~$w \in \mathcal{W}$. Hence, the strict convexity for  $\overline{\mathcal{G}}_k(\alpha, w, \mu, \varepsilon) $ is established.
	\end{proof}

	\begin{remark} \rm
	$(i)$	A similar proof used in Lemma~\ref{lemma: convexity} can be made to show that $\overline{\mathcal{G}}_k$ is strictly convex in $\mu.$
$(ii)$ For $\varepsilon >0$ and~$\alpha = 1/2$, the expected gain-loss function has the minimum at $\mu_0 = 0$. For $\alpha \neq 1/2,$ then the minimum, call it $\mu_0$, requires to solve a nonlinear equation
$
e^{\frac{1}{k} \log \frac{\alpha}{1-\alpha}} = \frac{1 + w(\mu-\varepsilon)}{ 1 - w(\mu +\varepsilon)}.
$
In either case, the associated gain-loss function~$\overline{\mathcal{G}}_k(\alpha, w, \mu_0, \varepsilon) < 0$; see also example in Section~\ref{subsection: Mote-Carlo Simulations}.
Using the fact that $\overline{\mathcal{G}}_k(\cdot)$ is continuous and strictly convex in $\mu$,  Intermediate Theorem; e.g., see \cite{rudin1964principles}, indicates that there exist two critical points~$\mu_{\pm}$ such that~$\overline{\mathcal{G}}_k(\alpha, w, \mu_{\pm}, \varepsilon) = 0$. 
Therefore, strict convexity assures that for $\mu > \mu_+$ or $\mu < \mu_-$, $\overline{\mathcal{G}}_k(\alpha, w, \mu_{\pm}, \varepsilon) > 0$.
However, obtaining the two critical points $\mu_{\pm}$ analytically require solving high-order nonlinear functions, which is intractable in general. In the following theorem, we propose an asymptote approach for $\overline{\mathcal{G}}_k$ to determine a moderately conservative positive expectation property. 
	\end{remark}

%
%
%

	\begin{theorem}[Positive Expectation] \label{theorem: Robust Positive Expectation} 
		Consider the double linear policy with pair $(\alpha, w)$.  For all~$k>1$, the following statements hold true: 
		
		$(i)$ If $\varepsilon  = 0$, then any pair $(\alpha, w)  \in \{1/2\} \times \mathcal{W} \setminus \{0\}$ guarantees the positive expected cumulative trading gain-loss; i.e.,
		$
		\overline{\mathcal{G}}_k(1/2, w, \mu, \varepsilon) > 0
		$
		for all $\mu \neq 0$. Moreover,  if $\mu =0$, we have $	\overline{\mathcal{G}}_k(1/2, w, \mu, \varepsilon) = 0.$
		
		$(ii)$ If $\varepsilon>0$, then any pair $(\alpha, w) \in (0,1) \times \mathcal{W} \setminus \{0\}$ guarantees a  positive expected cumulative trading gain-loss; i.e.,
		$
		\overline{\mathcal{G}}_k(\alpha, w, \mu, \varepsilon) > 0
		$
		for all $\mu > \mu_+$ or $\mu < \mu_-$ where 
		$$
		\mu_+ = \frac{1}{w} \left( e^{\frac{1}{k}\log \frac{1}{\alpha}} - 1 \right) + \varepsilon 
		$$
		and 
		$$
		\mu_- = \frac{1}{w} \left( 1-e^{\frac{1}{k} \log \frac{1}{1-\alpha}} \right) -\varepsilon.
		$$
	\end{theorem}
	
	\begin{proof} 
		To prove part~$(i)$, take $\varepsilon =0$ and fix $k>1$. Take the pair $(\alpha, w) \in \{1/2\} \times \mathcal{W} \setminus \{0\}$.  Lemma~\ref{lemma: expected cumulative gain or loss and variance} tells us that
		\[
		\overline{\mathcal{G}}_k(1/2, w, \mu, 0) = \frac{V_0}{2}\left( {{{\left( {1 + w\mu } \right)}^k} + {{\left( {1 - w\mu } \right)}^k} - 2} \right).
		\]
		If $\mu =0$, it is trivial to see that $\overline{\mathcal{G}}_k(1/2, w, 0, 0)=0.$
		On the other hand, if~$\mu \neq 0$, then $w\mu \neq 0$ for all $w \in \mathcal{W}\setminus \{0\}$. 
		By virtue of the fact that $(1+x)^k + (1-x)^k > 2$ for all~$x \neq 0$ and~$k > 1$, 
		the desired positivity is guaranteed by~$x := w \mu$. 

		To prove part~$(ii)$, we take $\varepsilon >0 $ and
		fix $k>1$. 
		Take a pair $(\alpha, w) \in  (0,1) \times \mathcal{W} \setminus \{0\}$.  Lemma~\ref{lemma: expected cumulative gain or loss and variance} tells us that
		$
		\overline{\mathcal{G}}_k(\alpha, w, \mu, \varepsilon) = V_0 f_k(\alpha, w,\mu, \varepsilon)
		$
		where 
		\begin{align*}
			f_k(\alpha, w,\mu, \varepsilon) 
			&:=\alpha \left( 1 + w (\mu- \varepsilon) \right)^k\\
			&+ (1-\alpha)\left( 1 - w (\mu + \varepsilon)  \right)^k - 1. 
		\end{align*}
		Since $V_0>0$, without loss of generality, we set~$V_0:=1$.
		Note that 
		$
		 f_k(\alpha, w,\mu, \varepsilon)  \geq \alpha \left( 1 + w (\mu- \varepsilon) \right)^k - 1
		 $
		  and  $f_k(\alpha, w, \mu, \varepsilon)$ is continuous in $\mu$. 
		The asymptote~$\alpha \left( 1 + w (\mu- \varepsilon) \right)^k - 1$ is continuous and monotonically increasing in $\mu$. 
		By solving the zero crossing root for the asymptote; i.e., solving~$\alpha \left( 1 + w (\mu- \varepsilon) \right)^k - 1:=0	$
		for~$\mu$ yields
		\[
		\mu  = \frac{1}{w} \left( e^{\frac{1}{k}\log \frac{1}{\alpha}} - 1 \right) + \varepsilon :=\mu_+
		\]  
		and $\mu_+ >0.$
		Hence, for $\mu > \mu_+$, we have $f_k(\alpha, w, \mu, \varepsilon) >0.$
		Similarly, it is also readily verified that
		$$
		f_k(\alpha, w,\mu, \varepsilon)  \geq (1-\alpha)\left( 1 - w (\mu + \varepsilon)  \right)^k - 1
		$$
		and the asymptote  $(1-\alpha)\left( 1 - w (\mu + \varepsilon)  \right)^k - 1$ is again continuous and monotonically decreasing in $\mu$. 
		Hence, solving~$(1-\alpha)\left( 1 - w (\mu + \varepsilon)  \right)^k - 1:=0$ yields
		\[
		\mu    =\frac{1}{w} \left( 1-e^{\frac{1}{k} \log \frac{1}{1-\alpha}} \right) -\varepsilon := \mu_- 
		\]  and $\mu_- <0$.
		Therefore, for $\mu < \mu_-$, $f_k(\alpha, w, \mu, \varepsilon) >0$.
		It is also readily verified that $\mu_- < 0 <  \mu_+$.
		Hence, in combination with the results above, we conclude that the desired positivity~$
		f_k(\alpha, w, \mu, \varepsilon) >0$ holds 
		for $ \mu >  \mu_+$ or~$ \mu <  \mu_- .$
	\end{proof}

	\begin{remark} \rm 
		The two critical points $\mu_{\pm}$ obtained in Theorem~\ref{theorem: Robust Positive Expectation} relies on zero crossing roots of the asymptote of $\overline{\mathcal{G}}_k$. Hence, it gives a conservative criterion for assuring the positive expectation. 
		In addition, if $k \to \infty$, then~$\mu_+, \mu_- \to \varepsilon, -\varepsilon$, respectively. When $k \to \infty$ and~$\varepsilon = 0$, part~$(ii)$ in Theorem~\ref{theorem: Robust Positive Expectation} reduces to part~$(i)$ in the same theorem.
	\end{remark}

	\subsection{Cash-Financing Lemma}
	As mentioned in Remark~\ref{remark: admissible set}, the double linear policy with pair $(\alpha, w) \in [0,1] \times \mathcal{W}$ assures a cash-financed trade.

	\begin{lemmarep}[Cash-Financing] \label{lemma: Cash-Financing}
		The double linear policy with a pair~$(\alpha, w) \in [0,1] \times \mathcal{W}$ satisfies the cash-financing condition. That is,
		$
		| \pi(k) |  \leq V(k)
		$ for all $k\geq 0$ with probability one.
	\end{lemmarep}
	
		\textit{Proof.} See Appendix.
		
	\begin{proof} 
		With $V_0>0$, fix $(\alpha, w) \in [0,1] \times \mathcal{W}$. For $k=0$,
		we begin by noting that $\pi_L(0) = w V_L(0) = \alpha wV_0$ and $\pi_S(0) = -wV_S(0) = -(1-\alpha) wV_0$. 
		Thus, observe that
		\begin{align*}
		|\pi(0)| &= |\pi_L(0) + \pi_S(0)| \\
		&\leq |\pi_L(0)| + |\pi_S(0)| = (2\alpha  - 1)wV_0 \leq V_0
		\end{align*}
	where the first inequality follows from the triangle inequality.
		Fix~$k \geq 1$. We have $\pi(k)  = \pi_L(k) +\pi_S(k)$ where $\pi_L(k) = wV_L(k)$ and $\pi_S(k) = -wV_S(k)$.
		Since~$w \in \mathcal{W}$, applying Lemma~\ref{lemma: survivability}, it is readily verified that~$V_L(k), V_S(k) \geq 0$ for all~$k\geq 1$ with probability one. 
		\begin{align*}
			|\pi(k)| 
			&= |\pi_L(k) + \pi_S(k)|\\
			& \leq w |V_L(k) + V_S(k)| \\
			&= w V(k) \leq V(k)
		\end{align*}
		where the last inequality holds since $w \in \mathcal{W} = [0, w_{\max}]$ and $w_{\max} \leq \frac{1}{1+\varepsilon} \leq 1$.  
	\end{proof}


	\section{Illustrative Examples} \label{SECTION: Illustrative Examples}
	This section provides two examples to illustrate our double linear policy scheme. The first example examines the positive expectation by carrying large numbers of Monte-Carlo simulations for prices modeled by geometric Brownian motion with jumps. The second example is for illustrating the trading performance using cryptocurrency historical data.
	
\subsection{Monte-Carlo Simulations:  GBM with Jumps} \label{subsection: Mote-Carlo Simulations}
For $t \in [0, T]$,	consider a stock whose price $S(\cdot)$ follows a geometric Brownian motion (GBM) with jump model; see~\cite{etheridge2002course}; i.e.,
\[
S(t) = S_0 \exp \left(\left(\mu^* - \frac{1}{2}\sigma^{*2} \right)t + \sigma^* W(t)\right)(1-\delta)^{N(t)}
\]
where $\{W(t)\}_{t\ge 0}$ is a standard Wiener process and $\{N(t)\}_{t \ge 0}$ is  a standard Poisson process with 
$P(N(t) = k) =\frac {(\lambda t)^k}{k!}e^{-\lambda t}$ that is independent of $W(t)$, $\mu^*$ is the {\it drift} constant, $\sigma^* > 0$ is the {\it volatility} constant, $\lambda$ is the average rate of the jump occurs for the process,  and~$\delta \in (0,1)$ denotes that for each jump reduces the stock price by a fraction $\delta$.  

	Taking a basic period length $\Delta t:=1/252$ and $T:=1$, we simulate the stock behavior for one year with an annualized\footnote{For a one-year $252$ trading days,  annualized drift rate and annualized volatility can be approximated by the daily drift rate and daily volatility; i.e.,~$\mu^* = 252 \mu$ and $\sigma^* = \sqrt{252}\sigma$.} drift rate $\mu^* \in [-0.9, 0.9]$ and annualized volatility~$\sigma^* := 2 |\mu^*| Z$ with $Z$ being a uniformly distributed random variable on interval~$[0,1]$, jump intensity $\lambda = 0.1$ with jump size~$\delta = 0.05$. 	
	For each $\mu^* \in [-0.9 : 0.01 : 0.9]$ with incremental values $0.01$ and $\sigma^*$, we generate $10,000$ GBM with jump sample paths, hence, a total of $1,810,000$ paths;
	 see also \cite{luenberger2013investment} for details on this topic.
	
 	With $\varepsilon = 0.01\%$ and $V_0:=\$1$, we consider a double linear policy with pair $(\alpha, w) \in [0,1] \times \{1/(1+\varepsilon)\}$.  
	Then we evaluate the associated expected trading gain-loss functions versus various~$\mu^*$; see Figure~\ref{fig:robustgainlossm10000}, where the red-colored dotted lines are obtained by Lemma~\ref{lemma: expected cumulative gain or loss and variance}, and the gray area indicates the negative expected gain-loss. 
	From the figure, for each choice of $\alpha$, we see that the Monte-Carlo results are consistent with the theory. In addition, it indicates that the desired expected positivity occurs after $\mu >  \mu^*_+$ and $\mu < \mu^*_-$ for some $\mu_+^*$ and~$\mu_-^*$. Similar patterns are also found for different choices of $\varepsilon, \lambda, \delta,$ and  $\sigma$.

\begin{figure}[h!]
	\centering
	\includegraphics[width=0.8 \linewidth]{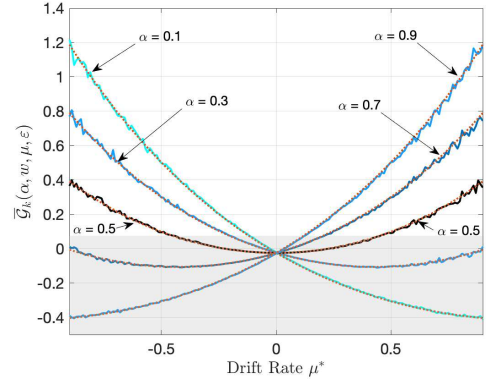}
	\caption{Gain-Loss Function Versus $\mu^* \in [-0.9, 0.9]$}
	\label{fig:robustgainlossm10000}
\end{figure}

%
%
%

\subsection{Backtesting with Cryptocurrency} \label{subsection: trading with historical data}
In this example, we apply the double linear policy to trade Bitcoin USD (BTC-USD) within the period from January~02,~2020 to August  01, 2022 with a total of $952$ days; see Figure~\ref{fig:bitcoinprice} for the corresponding prices trajectory. 
Compared to conventional stocks, Bitcoin prices are much more volatile and sensitive to regulatory and market events.
 In this case, maximum and minimum daily returns are~$X_{\max} \approx 0.1875$ and $X_{\min} \approx -0.3717$, respectively.
With initial account value of~$V_0 = \$100,000$, transaction costs\footnote{Generally, most cryptocurrency exchanges charge between 0\% and 1.5\% per trade, depending on the trading volume and whether it is buying or selling; e.g., see~\cite{kim2017transaction}. } $\varepsilon = 0.01\%$ and weight~$w= 1/4$, the trading performances in terms of gain-loss functions are shown in Figure~\ref{fig:bitcointrading}.
For higher transaction costs, say~$\varepsilon = 0.1\%$ with the same weight $w=1/4$, the trading performances  are shown in Figure~\ref{fig:bitcointradingk14}. 
From the figures, we see the effect of transaction costs, and one should choose  a proper allocation constant $\alpha$ of the fund when $\varepsilon>0.$

\begin{figure}[h!]
	\centering
	\includegraphics[width=.8\linewidth]{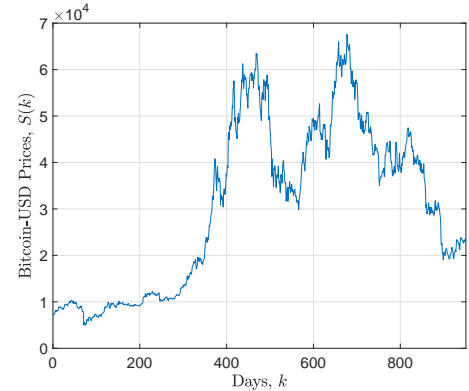}
	\caption{BTC-USD Daily Prices (January 02, 2020 -- August 01, 2022).}
	\label{fig:bitcoinprice}
\end{figure}

\begin{figure}[h!]
	\centering
	\includegraphics[width=.8\linewidth]{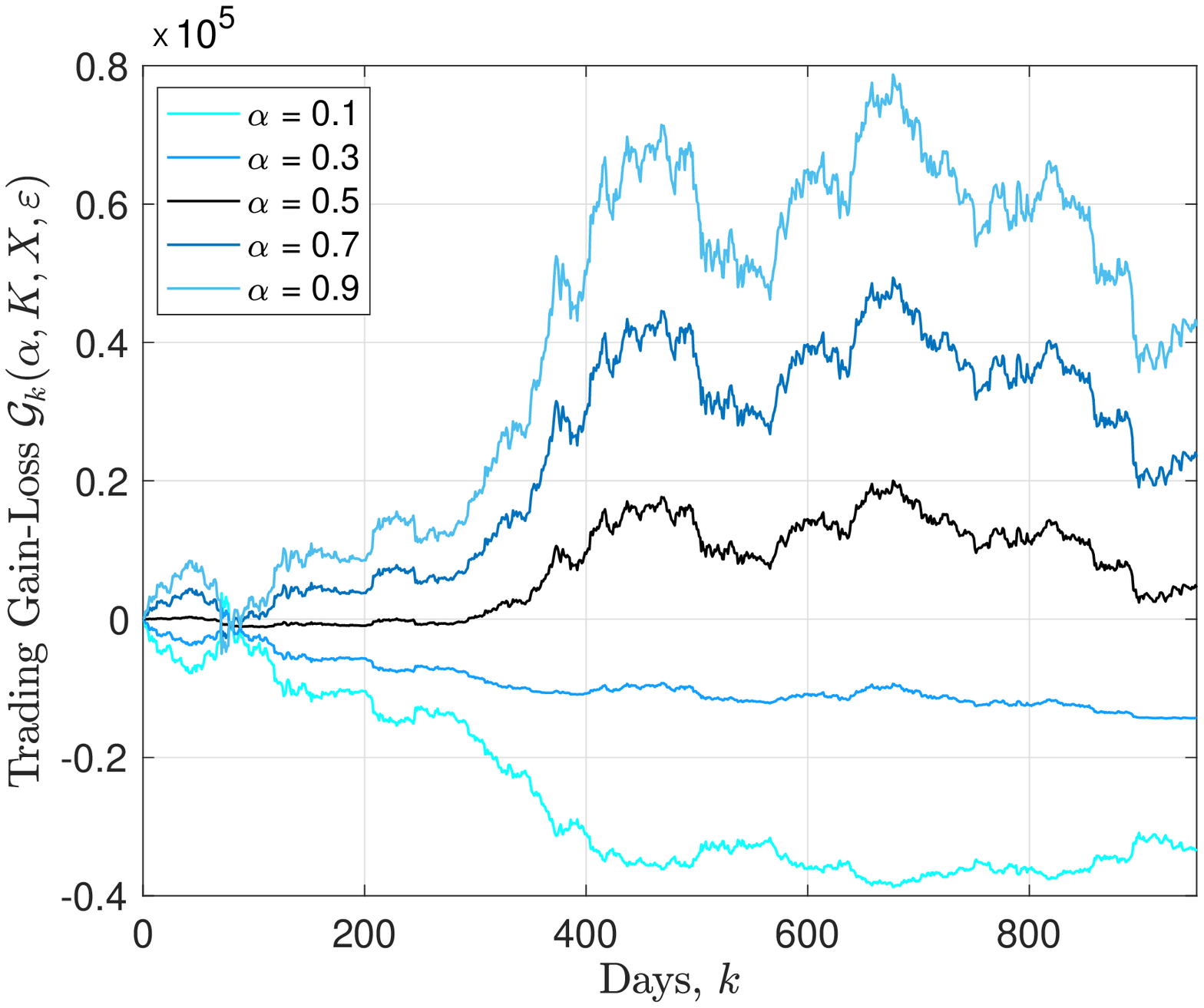}
	\caption{Trading Against BTC-USD with $\varepsilon = 0.01\%$.}
	\label{fig:bitcointrading}
\end{figure}

\begin{figure}[h!]
	\centering
	\includegraphics[width=0.8\linewidth]{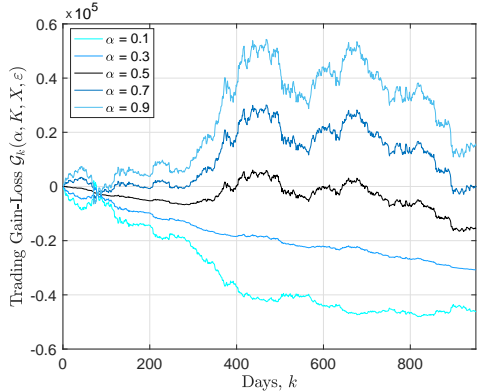}
	\caption{Trading Against BTC-USD with $\varepsilon = 0.1\%$.}
	\label{fig:bitcointradingk14}
\end{figure}

\section{Conclusion} \label{SECTION: Conclusion}
In this paper, we propose a new class of double linear policy schemes in an asset trading scenario with nonzero transaction costs. 
We  indicated that the desired positive expectation result might disappear when the costs are involved.
 In addition, we quantify under what conditions the positive expectations can still be preserved.
We also validate our theory via heavy Monte-Carlo simulations using  GBM with jumps and provide a backtesting example using Bitcoin-USD historical data. 

Regarding future research, it would be of interest to determine an ``optimal" pair $(\alpha, w)$ that also preserves desired positivity; see \cite{maroni2019robust} for some initial studies along this line.	 
The other interesting direction is to consider a multi-asset portfolio case; see \cite{primbs2018pairs, deshpande2020simultaneous} for a preliminary extension to pair tradings.
	
	\bibliographystyle{ieeetr}
	\bibliography{refs}

\begin{thebibliography}{10}

\bibitem{barmish2008trading}
B.~R. Barmish, ``{On Trading of Equities: A Robust Control Paradigm},'' {\em
  IFAC Proceedings Volumes}, vol.~41, no.~2, pp.~1621--1626, 2008.

\bibitem{barmish2011performance}
B.~R. Barmish, ``{On Performance Limits of Feedback Control-based Stock Trading
  Strategies},'' in {\em Proceedings of the 2011 American Control Conference},
  pp.~3874--3879, 2011.

\bibitem{barmish2011arbitrage}
B.~R. Barmish and J.~A. Primbs, ``{On Arbitrage Possibilities via Linear
  Feedback in an Idealized Brownian Motion Stock Market},'' in {\em 2011 50th
  IEEE Conference on Decision and Control and European Control Conference},
  pp.~2889--2894, 2011.

\bibitem{barmish2015new}
B.~R. Barmish and J.~A. Primbs, ``{On a New Paradigm for Stock Trading via a
  Model-Free Feedback Controller},'' {\em IEEE Transactions on Automatic
  Control}, vol.~61, no.~3, pp.~662--676, 2015.

\bibitem{baumann2016stock}
M.~H. Baumann, ``{On Stock Trading via Feedback Control When Underlying Stock
  Returns are Discontinuous},'' {\em IEEE Transactions on Automatic Control},
  vol.~62, no.~6, pp.~2987--2992, 2016.

\bibitem{baumann2017simultaneously}
M.~H. Baumann and L.~Gr{\"u}ne, ``{Simultaneously Long Short trading in
  Discrete and Continuous Time},'' {\em Systems \& Control Letters}, vol.~99,
  pp.~85--89, 2017.

\bibitem{primbs2017robustness}
J.~A. Primbs and B.~R. Barmish, ``{On Robustness of Simultaneous Long-Short
  Stock Trading Control with Time-Varying Price Dynamics},'' {\em
  IFAC-PapersOnLine}, vol.~50, no.~1, pp.~12267--12272, 2017.

\bibitem{malekpour2018generalization}
S.~Malekpour, J.~A. Primbs, and B.~R. Barmish, ``{A Generalization of
  Simultaneous Long-Short Stock Trading to PI Controllers},'' {\em IEEE
  Transactions on Automatic Control}, vol.~63, no.~10, pp.~3531--3536, 2018.

\bibitem{malekpour2016stock}
S.~Malekpour and B.~R. Barmish, ``{On Stock Trading Using a Controller with
  Delay: The Robust Positive Expectation Property},'' in {\em 2016 IEEE 55th
  Conference on Decision and Control (CDC)}, pp.~2881--2887, 2016.

\bibitem{deshpande2020simultaneous}
A.~Deshpande, J.~A. Gubner, and B.~R. Barmish, ``{On Simultaneous Long-Short
  Stock Trading Controllers with Cross-Coupling},'' {\em IFAC-PapersOnLine},
  vol.~53, no.~2, pp.~16989--16995, 2020.

\bibitem{o2020generalized}
J.~D. O'Brien, M.~E. Burke, and K.~Burke, ``{A Generalized Framework for
  Simultaneous Long-Short Feedback Trading},'' {\em IEEE Transactions on
  Automatic Control}, 2020.

\bibitem{maroni2019robust}
G.~Maroni, S.~Formentin, and F.~Previdi, ``{A Robust Design Strategy for Stock
  Trading via Feedback Control},'' in {\em 2019 18th European Control
  Conference (ECC)}, pp.~447--452, 2019.

\bibitem{zhang2001stock}
Q.~Zhang, ``{Stock Trading: An Optimal Selling Rule},'' {\em SIAM Journal on
  Control and Optimization}, vol.~40, no.~1, pp.~64--87, 2001.

\bibitem{tie2018optimal}
J.~Tie, H.~Zhang, and Q.~Zhang, ``{An Optimal Strategy for Pairs Trading Under
  Geometric Brownian Motions},'' {\em Journal of Optimization Theory and
  Applications}, vol.~179, no.~2, pp.~654--675, 2018.

\bibitem{deshpande2018generalization}
A.~Deshpande and B.~R. Barmish, ``{A Generalization of the Robust Positive
  Expectation Theorem for Stock Trading Via Feedback Control},'' in {\em 2018
  European Control Conference (ECC)}, pp.~514--520, 2018.

\bibitem{barmish2021feedforward}
B.~R. Barmish, J.~A. Primbs, and S.~Warnick, ``{On Feedforward Stock Trading
  Control Using a New Transaction Level Price Trend Model},'' {\em IEEE
  Transactions on Automatic Control}, vol.~67, no.~2, pp.~902--909, 2021.

\bibitem{luenberger2013investment}
D.~G. Luenberger, {\em {Investment Science}}.
\newblock Oxford University Press, 2013.

\bibitem{bodie2009investments}
Z.~Bodie, A.~Kane, and A.~J. Marcus, {\em Investments}.
\newblock McGraw-Hill Education, 2013.

\bibitem{hsieh2019positive}
C.-H. Hsieh, B.~R. Barmish, and J.~A. Gubner, ``{On Positive Solutions of a
  Delay Equation Arising When Trading in Financial Markets},'' {\em IEEE
  Transactions on Automatic Control}, vol.~65, no.~7, pp.~3143--3149, 2019.

\bibitem{hsieh2020necessary}
C.-H. Hsieh, ``{Necessary and Sufficient Conditions for Frequency-Based Kelly
  Optimal Portfolio},'' {\em IEEE Control Systems Letters}, vol.~5, no.~1,
  pp.~349--354, 2020.

\bibitem{hsieh2022generalization}
C.-H. Hsieh, ``{Generalization of Affine Feedback Stock Trading Results to
  Include Stop-Loss Orders},'' {\em Automatica}, vol.~136, p.~110051, 2022.

\bibitem{cvitanic2004introduction}
J.~Cvitanic and F.~Zapatero, {\em {Introduction to the Economics and
  Mathematics of Financial Markets}}.
\newblock MIT press, 2004.

\bibitem{rudin1964principles}
W.~Rudin, {\em {Principles of Mathematical Analysis}}.
\newblock McGraw-hill New York, 1964.

\bibitem{etheridge2002course}
A.~Etheridge, {\em {A Course in Financial Calculus}}.
\newblock Cambridge University Press, 2002.

\bibitem{kim2017transaction}
T.~Kim, ``{On the Transaction Cost of Bitcoin},'' {\em Finance Research
  Letters}, vol.~23, pp.~300--305, 2017.

\bibitem{primbs2018pairs}
J.~A. Primbs and Y.~Yamada, ``{Pairs Trading Under Transaction Costs Using
  Model Predictive Control},'' {\em Quantitative Finance}, vol.~18, no.~6,
  pp.~885--895, 2018.

\end{thebibliography}
	

\end{document}